\documentclass[11pt,a4paper]{amsart}

\usepackage{amssymb}
\usepackage{amscd}
\usepackage[OT4]{fontenc}
\usepackage{listings}

\usepackage{bbm}
\usepackage{tikz-cd}
\usepackage{diagrams}
\diagramstyle[centredisplay,nohug]

\newtheorem{thm}{Theorem}[section]
\newtheorem{theorem}[thm]{Theorem}

\newtheorem{lemma}[thm]{Lemma}
\newtheorem{proposition}[thm]{Proposition}
\newtheorem{corollary}[thm]{Corollary}

\theoremstyle{definition}

\newtheorem{remark}[thm]{Remark}
\newtheorem{example}[thm]{Example}

\newtheorem{thevarthm}[thm]{\varthmname}

\newenvironment{varthm*}[1]{\trivlist\item[]{\bf #1.}\it}{\endtrivlist}

\renewcommand\ge{\geqslant}
\renewcommand\geq{\geqslant}

\renewcommand\leq{\leqslant}

\newcommand\be{\begin{eqnarray*}}
\newcommand\ee{\end{eqnarray*}}

\newcommand\C{\mathbb C}
\newcommand\Z{\mathbb Z}

\newcommand\E{\mathbb E}
\renewcommand\P{\mathbb P}

\newcommand\T{\mathbb T}
\newcommand\G{\mathbb G}

\newcommand\calc{{\mathcal C}}
\newcommand\tcalc{\widetilde{\calc}}
\newcommand\calo{{\mathcal O}}

\newcommand\call{{\mathcal L}}
\newcommand\calm{{\mathcal M}}

\newcommand\calt{{\mathcal T}}

\newcommand\alphahat{\widehat{\alpha}}

\newcommand\newop[2]{\def#1{\mathop{\rm #2}\nolimits}}
\newop\log{log}
\newop\ord{ord}
\newop\Gal{Gal}
\newop\SL{SL}
\newop\Bl{Bl}
\newop\mult{mult}
\newop\mass{mass}
\newop\div{div}
\newop\codim{codim}
\newop\sing{sing}
\newop\vdim{vdim}
\newop\edim{edim}
\newop\Ass{Ass}
\newop\size{size}
\newop\virdim{dim_{vir}}
\newop\expdim{dim_{exp}}
\newcommand\eqnref[1]{(\ref{#1})}

\newcommand\restr[1]{\big|_{#1}}

\newcommand\beginproof[1]{\trivlist\item[\hskip\labelsep{\em #1.}]}
\renewcommand\proof{\beginproof{Proof}}
\newcommand\proofof[1]{\beginproof{Proof of #1}}
\def\endproof{\hspace*{\fill}\endproofsymbol\endtrivlist}

\def\endproofsymbol{\frame{\rule[0pt]{0pt}{6pt}\rule[0pt]{6pt}{0pt}}}

\def\keywordname{{\bfseries Keywords}}%
\def\keywords#1{\par\addvspace\medskipamount{\rightskip=0pt plus1cm
\def\and{\ifhmode\unskip\nobreak\fi\ $\cdot$
}\noindent\keywordname\enspace\ignorespaces#1\par}}
\def\subclassname{{\bfseries Mathematics Subject Classification
(2000)}\enspace}
\def\subclass#1{\par\addvspace\medskipamount{\rightskip=0pt plus1cm
\def\and{\ifhmode\unskip\nobreak\fi\ $\cdot$
}\noindent\subclassname\ignorespaces#1\par}}

\begin{document}

\author{M.~Dumnicki, B.~Harbourne, J.~Ro\'e, T.~Szemberg, H.~Tutaj-Gasi\'nska}
\title{Unexpected surfaces singular on lines in $\P^3$}


\maketitle

\begin{abstract}
   We study linear systems of surfaces
   in $\P^3$ singular along general lines.
  Our purpose is to identify and classify special systems of such surfaces, i.e.,
  those nonempty systems where the
  conditions imposed by the multiple lines are not independent.
  We prove the existence of four surfaces arising as (projective) linear systems
  with a single reduced member, which numerical experiments had suggested must exist.
  These are \emph{unexpected} surfaces and we expect that our list is complete,
  i.e. it contains all special linear systems of affine dimension $1$, whose projectivisation has one, reduced and irreducible member.

    As an application we find upper bounds for Waldschmidt constants
   along certain sets of general lines.
\end{abstract}
  
\subclass{14C20; 14E05}


\section{Introduction}
   The study of linear systems of hypersurfaces in complex projective spaces
   with assigned base points of given multiplicity is a classical and
   central problem in algebraic geometry; see, e.g. \cite{AleHir95}, \cite{Iar97},
   \cite{CilMir98}, \cite{Har10}.

   In the last few years this problem has been generalized to linear systems of hypersurfaces
   with assigned base loci consisting of linear subspaces of higher dimension \cite{GHVT13,DHST14}.
   Conjectures such as \cite[Conjecture 5.5]{GHVT13} and \cite[Conjectures A, B and C]{DHST14}
   suggest that the asymptotic behavior of such linear systems in the case of
   linear subspaces of higher dimension is
   similar to the case of points.
   In particular, after fixing sufficiently many sufficiently general linear
   subspaces, it is expected that the conditions imposed on forms by vanishing along these subspaces
   will be independent. (For efficiency we slightly abuse terminology
   by saying that $r$ homogeneous linear equations in a vector
   space of dimension $s$ are \emph{independent}
   if the subspace of solutions has dimension either $s-r$ or $0$.)
   On the other hand,
   it is interesting and important to understand \emph{special} linear systems,
   i.e., nonempty systems for which the imposed conditions are dependent, or,
   equivalently, nonempty linear systems whose dimension is greater than
   expected from a naive conditions count.

 In the classical setup of assigned base points in the case of generic points in $\P^2$, the well-known
   Segre-Harbourne-Gimigliano-Hirschowitz Conjecture (SHGH for short) \cite{Seg61,Har84,Gi87,Hirsh89},
   provides a complete (yet still conjectural) explanation of all
   special linear systems of plane curves; see \cite{Cil00} for
   a nice survey and \cite{CHMR13} for an account on recent progress.
   In the case of $\P^3$, there is an analogous conjecture due to
   Laface and Ugaglia \cite{LafUga05, LafUga06}.
   In addition, there is
   a very nice partial result due to Brambilla,
   Dumitrescu and Postinghel \cite{BDP14} valid for points
   in projective spaces of arbitrary dimension.

Recently a new path of research has been opened in
\cite{CHMN16} and \cite{HMNT18}, where the authors introduce the notion
of unexpected hypersurfaces; see also \cite{Szp18} for interesting connections with Lefschetz Properties and
hypo-osculating varieties.
With respect to the notation introduced in the next section,
the hypersurfaces in $\call=\call_d(m_1,\dots,m_s)$ being unexpected
just means $\call$ is special (defined below).
	
   In the present note we focus
   on the classification of special linear systems with base loci
   assigned along positive dimensional linear subspaces in the
   simplest nontrivial situation, i.e., we consider linear systems of surfaces
   in $\P^3$ with vanishing conditions imposed along general lines.
   Our main result is Theorem \ref{ex:Singular list}.

\section{Notation and basic properties}
   In our context of imposed base \emph{lines}, we will use
   the same notation customarily used for
   linear systems with imposed base \emph{points}. This has
   the advantage of working with familiar notation but with no danger
   of confusion since we clearly flag appearances of assigned base points.

   Thus, $\call=\call_d(m_1,\dots,m_s)$ denotes the linear system of surfaces
   of degree $d$ in $\P^3$ passing through $s$ general lines (hence the lines are
   in particular disjoint) with assigned multiplicities $m_1,\dots,m_s$.
   If $d<m_i$ for some $i$, then clearly $\call_d(m_1,\dots,m_s)=\varnothing$, so we will
   always assume that $d\geq \max(m_1,\ldots,m_s)$.
   As is customary, if the multiplicities are repeated, then we abbreviate the notation in a
   natural way. For example $\call_d(m^{\times s})$ denotes a linear
   system of surfaces of degree $d$ with $s$ lines of the same multiplicity $m$.

  Let $c_{m,d}$ be the number of conditions
   which vanishing to order $m$ along a line in $\P^3$
   imposes on forms of degree $d\geq m$. This number is well-known and is worked out in \cite[Lemma A.2(c)]{DHST14}:
   $$c_{m,d}=\sum_{i=0}^{m-1}(d+1-i)(i+1)=\frac{1}{6}m(m+1)(3d+5-2m).$$
   For the convenience of the reader we provide explicit formulas for a few initial values of $m$:
   $$c_{1,d}=d+1,\;\; c_{2,d}=3d+1,\;\; c_{3,d}=6d-2,\;\; c_{4,d}=10d-10,\;\; c_{5,d}=15d-25.$$
   The \emph{virtual} dimension of $\call$ is therefore
   \begin{equation}\label{eq:virtdim}
   \virdim(\call)=\binom{d+3}{3}-\frac16\sum_{i=1}^{s}m_i(m_i+1)(3d+5-2m_i).
   \end{equation}
   Note that we use here \emph{affine} dimension; i.e., the dimension of
   the vector space of forms defining the surfaces in the linear system.
   As usual, the \emph{expected} dimension of $\call$ is
   $$\expdim(\call)=\max\left\{\virdim(\call),0\right\}.$$
   If the actual dimension of $\call$ is larger than the expected
   dimension, then we say that $\call$ is \emph{special}. In this note we are
   primarily interested in special systems.

\section{Reasons for speciality}
   The conditions imposed by general lines with multiplicity $1$ are
   always independent. In other words, linear systems of the type
   $$\call_d(1^{\times s})$$
   are always non-special. While general assigned base points of multiplicity 1
   trivially impose independent conditions, in the case of base lines of multiplicity 1
   this is a non-trivial result due to Hartshorne and Hirschowitz \cite{HarHir81}.

   Thus as with points it requires lines of higher multiplicities in order to get a special system.
   There are two easy ways to construct such special linear systems.
   We discuss them in the following two examples.

\begin{example}[Multiples of non-special systems]\label{ex:multiple}
   Let $\call=\call_2(1,1,1)$. Then, by the Hartshorne and
   Hirschowitz result, $\call$ is non-special with $\dim(\call)=1$, so the (projective) linear
   system $\call$ contains a unique quadric $Q$. On the other hand, for
   $\calm=2\call=\call_4(2^{\times 3})$ we have
   $$\virdim(\calm)=-4, \mbox{ hence } \expdim(\calm)=0$$
   but of course $\dim(\calm)\geq 1$ since $2Q$ is in $\calm$, and it is not hard to verify that $\dim(\calm)=1$
   so $2Q$ is the only member of the projectivisation of $\calm$.
   In fact, all linear systems $\call_{2m}(m^{\times 3})$ with $m\ge2$ are
   special of affine dimension 1.

   Another instance of this is given by the linear system $\call_{3}(1^{\times 4})$
   of cubics containing 4 general lines. By considering
   quadrics through 3 of the 4 lines with a plane containing the fourth
   we see that $\call_{3}(1^{\times 4})$ is nonempty.

   Since $\call_{3m}(m^{\times 4})$
   has negative virtual dimension if (and only if) $m\ge 8$,
   yet is a multiple of the nonempty system
   $\call_{3}(1^{\times 4})$,
   it is nonempty and hence special for $m\geq8$.
   Note for 4 general lines that there
   are two lines which meet each of the 4 transversally. Thus both of these
   transversal lines are
   in the base locus of $\call_{3}(1^{\times 4})$, and by considering
   quadrics through 3 of the 4 lines with a plane containing the fourth,
   it is not hard to see that the base locus is precisely the two transversals
   (see \cite[Example 3.4.3]{D11}). Since the base locus of $\call_{3}(1^{\times 4})$ has no divisorial
   components, the general member of $\call_{3m}(m^{\times 4})$ is reduced.
   In particular, these give special linear systems whose general members
   are reduced. This is in contrast to $\call_{2m}(m^{\times 3})$, and to the assertion of the SHGH Conjecture for special
   linear systems of curves in the plane (which says that every special system has
   a multiple curve in its base locus).
\end{example}
   Instead of taking multiples of a fixed system, one may add some distinct systems.
\begin{example}[Unions of non-special systems]\label{ex:union}
   Let $\call=\call_8(3^{\times 4})$. Then
   $$\virdim(\call)=-19, \mbox{ hence } \expdim(\call)=0$$
   but this system is non-empty, as it contains the element consisting of the union of four quadrics
   (hence it is reduced), each of which vanishes
   along three of the four given general lines. In fact, it can be easily checked
   that this is the only element of the system.
\end{example}

   In contrast to the previous examples, it is possible for a special system
  to be reduced, irreducible and of affine dimension one.
   An extensive computer search using Singular \cite{Singular} exhibited four such special linear systems, listed in Theorem \ref{ex:Singular list}.
   Since in each case the  projectivisation of the system contains a unique, irreducible member,
   the speciality cannot be explained along the lines of Examples
   \ref{ex:multiple} and \ref{ex:union}.
   We will see in section \ref{sec:Waldschmidt} how the existence of systems B) and C) is related
   to some important asymptotic invariants of homogeneous ideals.

\begin{theorem}\label{ex:Singular list}
   The following systems are special of (affine) dimension $1$:
   \begin{itemize}
   \item[A)] $\call_{10}(3^{\times 4},1^{\times 5})$;
   \item[B)] $\call_{12}(4,3^{\times 5})$;
   \item[C)] $\call_{12}(3^{\times 6},2)$;
   \item[D)] $\call_{20}(6^{\times 5},1)$.
   \end{itemize}
   Thus there is a single surface of the given degree
   vanishing to given order along the given number of \emph{general} lines.
\end{theorem}

The systems above are special as long as they are effective.
   Note that the Singular computation which led to the systems
   in Theorem \ref{ex:Singular list} works with \emph{random}
   rather than \emph{general} lines, and so it gives strong evidence, but
   not proof, of the effectivity of the given systems.
   This Theorem is proved in subsection \ref{sec:first part} and in section
   \ref{sec:7 lines}.

   \begin{remark}
   In cases A), B) and D) of Theorem \ref{ex:Singular list} the unexpected surfaces are rational.
   However in case C) the unexpected surface is of general type by Corollary \ref{cor:general-type}.
   This is quite surprising in view of the special effect varieties program proposed by Bocci \cite{Boc05}.
   \end{remark}

\section{Cremona transformations of $\P^3$}\label{sec:Cremona}
   Cremona transformations have long played a prominent role in the
   study of special linear systems in the plane. It turns out
   that this is also a useful approach in our situation.
   We recall here two birational transformations of $\P^3$
   which best suit our needs. We provide some background for non-experts.

\subsection{A cubo-cubic transformation}
   First, there is the cubo-cubic transformation $\calc:\P^3\dashrightarrow\P^3$
   induced by the linear system of cubics vanishing to order $1$ along $4$ general lines,
   for which both the image and inverse image of a general plane is a cubic surface (hence the name
   cubo-cubic). This compares to the more familiar
   plane quadratic Cremona transformation induced by the linear system of conics vanishing
   to order $1$ at $3$ general points (which by analogy is called a quadro-quadric
   transformation).

  %

      \begin{figure}[h]
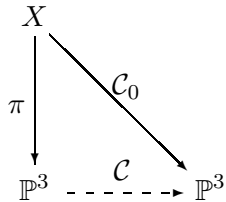

   \begin{diagram}
      X &  & \\                                   %
      \dTo^{\pi} & \rdTo\;\raise.1in\hbox{$\calc_0$} & \\			
      \P^3 & \rDashto^{\calc} & \P^3\\			
   \end{diagram}
   \caption{The cubo-cubic transformation.}
   \label{diagram1}
   \end{figure}

   As mentioned in Example \ref{ex:multiple}, the base locus of the linear system of cubic surfaces vanishing on
   4 general lines consists of 6 lines: the 4 general lines,
   and two additional lines which are transversal to the first 4; i.e., both tranversal lines intersect each of the 4 general lines
   (see \cite[Example 3.4.3]{D11}).

       Let $\pi:X\to\P^3$ be the morphism obtained by first blowing up $\P^3$ along each of four general lines $l_1,\ldots,l_4$
   (whose exceptional divisors we denote by $E_1,\ldots,E_4$) and then blowing up
   the two transversal lines $t_1,t_2$ (whose exceptional divisors we denote by $T_1$ and $T_2$).
   One can show that $\pi$ resolves the indeterminacy of $\calc$, giving
      a morphism $\calc_0$, as shown in Figure \ref{diagram1}.

   Denote by $H$ the pullback via $\pi$ to $X$ of a plane in $\P^3$.
    By $\E$ we denote $E_1+\dots +E_4$ and by $\T$ we denote $T_1+T_2$.
   Thus $\calc_0$ is induced by the linear system of sections of $3H-\E-\T$.
   By $H'$ we denote the pullback   of a plane via $ \calc_0$, so $H'=3H-\E-\T$.

     The intersection product on $X$ is determined by
   \begin{equation}\label{eq:intersection product}
   H^3 =1, H E_i^2=H T_j^2=E_iT_j^2=-1 \mbox{ and }  -E_i^3 =T_j^3=2
   \end{equation}
   with all other monomial triple intersections being 0.
   As  $H^3=1$ and $H^2E_i=H^2T_i=0$ are clear,  we briefly explain the other intersections.
   Note that the blowup of a line in $\P^3$ is isomorphic to $\P^1\times\P^1$, thus sections of $H-E_i$ are disjoint so $H(H-E_i)^2=0$.
   Expanding this shows that $HE_i^2=-1$. Similarly expanding $(H-E_i)^3=0$ we get $E_i^3=-2$.

   Computing $T_j^3$ is more subtle.
   Consider $(H-T_1)^2T_1$. We will show that this intersection equals $4$,
   then expanding, we get that $T_1^3=2$.

   Each element of $|H-T_1|$ is the proper transform of a plane $h\subset \P^3$ containing $t_1$.
   The proper transform of $h$ is the blowup of $h$ in the four points $p_1,\ldots, p_4$ where the lines $l_i$ meet $h$.
   So if $A,B$ are different elements of $|H-T_1|$, then the restrictions $A\stackrel{\pi}{\to} a$, $B\stackrel{\pi}{\to} b$
   of $\pi$ to $A$ and $B$
   are the blowups of $p_1,\ldots, p_4$ on the planes $a$ and $b$. Thus, $A\cap B=e_{p_1}\cup\ldots \cup e_{p_4}$,
   where $e_{p_i}$ is the exceptional curve given by the blowup of $p_i$.
   Since $\#(e_{p_i}\cap T_1)=1$ and the intersection is transversal, we get $(H-T_1)^2T_1=4$.


   The morphism $\calc_0:X\to\P^3$ factorizes into an isomorphism of $X$ followed by a sequence of blowups of $\P^3$, as we now explain.
   There are four quadrics, say $q_1,\ldots,q_4$, where  $q_1$ is the unique quadric which passes through the lines
    $l_2,l_3,l_4$,  $q_2$ is the unique quadric which passes through $l_1,l_3,l_4$, etc. The proper transform by $\pi$ of $q_i$ is $Q_i$.
    The divisor class of $Q_i$ is $2H-\E+E_i-\T $. The image of $Q_i$ under $\calc_0$ is
    a line in $\P^3$, call it $l_i'$ (as the restriction of $3H-\E-\T$ to $Q_i$ is a $(1,0)$ class).
   The images of $T_1 $ and $T_2$ are lines $t_1',t_2'$ transversal to $l_1',\ldots,l_4'$. We note that $\pi(T_i)$ is  projection
   along one ruling of $T_i=\P^1\times\P^1$ and $\calc_0(T_i)$ is  projection along the other  ruling.
   In Figure \ref{diagram1b}, let $\pi_1':X_1'\to\P^3$ be the blowup of the lines $l_i'$ and let $\pi_2':X_2'\to X_1'$
   be the blowup of the proper transforms of the lines $t_i'$. By the universal property of blowups,
   ${\calc}_0$ lifts to ${\calc}_1$, and then ${\calc}_1$ lifts to ${\calc}_2$.

   We claim ${\calc}_2$ is an isomorphism. Because $H^2H'=3$, the
   inverse of $\calc$ is defined by a four dimensional family of cubic surfaces. Since the fibers over $l_i', t_j'$ are positive dimensional,
   the inverse $\calc'$ of $\calc$ is not defined on these six lines, so the base locus of $\calc'$ contains these six lines,
   and, as for ${\calc}$, this is the whole base locus and a morphism $\calc_0':X_2'\to\P^3$ resolving the indeterminacies of ${\calc}'$
   is obtained by blowing them up, $l_1',\ldots,l_4'$ first, and then $t_1', t_2'$. Arguing as before, $\calc_0'$ lifts to a morphism
   ${\calc}_2'$, and since ${\calc}$ and ${\calc}'$ are inverse birational maps, we see that ${\calc}_2$ and ${\calc}_2'$ are inverse morphisms.

\begin{figure}[h]
\begin{center}
   \begin{tikzcd}[column sep=huge]
      X \arrow[dd,"\pi"] \arrow[ddr,"{\calc}_0"] \arrow[dr,"{\calc}_1"] \arrow[r,"{\calc}_2"] & X_2' \arrow[d,"\pi_2'"]\\
   & X_1' \arrow[d,"\pi_1'"]\\	
      \P^3 \arrow[r,dashed,"{\calc}"] & \P^3\\		
   \end{tikzcd}
   \end{center}
   \caption{A factorization of the morphism ${\calc}_2'$ as blowups of the $l_i'$ and the $t_i'$.}
   \label{diagram1b}
\end{figure}

 The following transformation rules
 \begin{equation}\label{eq:trafo rules for cubo cubic}
   \begin{array}{ccl}
      H' & \mapsto & 3H -\E-\T,\\
      E_i' & \mapsto & 2H -\E+E_i-\T,	\\		
      T_i' & \mapsto & T_i
   \end{array}
   \end{equation}
   define a linear transformation $\Gamma_{\calc}: Cl(X_2')\to Cl(X)$ on the divisor class group of $X_2'$ (i.e., the free $\Z$-module generated by $H', E_1',\ldots,E_4',T_1',T_2'$).
   This transformation preserves all triple products and its matrix is its own inverse.
This is the map on the divisor class groups induced by pullback by the birational map $\calc$ in Figure \ref{diagram1b}.

\begin{remark}
The divisors $T_i$ play a role in defining $\calc$ but, as explained below in Remark \ref{SpecialByCremona},
we will obtain special systems by applying $\Gamma_{\calc}$ to divisor classes in which the $T_i$ do not occur as summands.
\end{remark}

\subsection{Todd transformation}
   There is another interesting transformation
   $\calt: \P^3\dashrightarrow \P^3$ which seems to go back
   to Todd \cite[Introduction]{Tod32}. It is given by the
   linear system of surfaces of degree $19$ vanishing to order $5$
   along $6$ general lines $l_1,\ldots,l_6$ in $\P^3$. We now summarize Todd's results.

    The geometry of this Cremona map can be analyzed similarly to the cubo-cubic case. The base locus of the
   linear system of surfaces of degree $19$ vanishing to order $5$ along six general lines consists of the six lines,
   the 30 transversal lines which are determined by subsets of four out of the six lines, and, as explained in \cite{Wake} the six twisted cubics that have the six lines
   as chords.
   Let $\pi:X\to \P^3$ be the composition of the blowing up of the source $\P^3$ in the set of base lines,
   followed by the blowing up in the additional 30 lines and then the 6 cubics.
   We denote the exceptional divisors by $E_1,\ldots,E_6$, $T_1,\ldots,T_{30}$,
   $G_1,\ldots,G_6$. We abbreviate as above $\E=E_1+\dots+E_6$, $\T=T_1+\dots+T_{30}$, $\G=G_1+\dots+G_6$.
   Then
   $\calt$ induces a birational morphism $\calt_0:X\to\P^3$ determined		%
   by the linear system $19H-5\E-\T-3\mathbb{G}$.  %
  We have the  commutative diagram \ref{diagram2}.
   \begin{figure}[h]
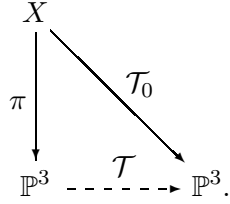

   	\begin{diagram}
   		X &  & \\
   		\dTo^{\pi} & \rdTo^{\calt_0} & \\
   		\P^3 & \rDashto^{\calt} & \P^3.
   	\end{diagram}
   	\caption{The Todd transformation.}
   	\label{diagram2}
   \end{figure}

   The inverse map is again given by
   surfaces of degree $19$ vanishing along a set of $6$  lines, $l_1',\ldots,l_6'$.
   Todd shows \cite[p. 59]{Tod32} that each base line $l_i$ is transformed into $D_i$, a (rational) surface   of degree 12.  Denote by $t:D_i\to l_i\cong\P^1$
   the restriction of the inverse map. Todd proves that the fibers of $t$ are quintic curves which meet
   five of the lines $l_1'\ldots l_6'$ in three points and the remaining line in four points. The map $t$ is in fact a morphism,
   and $D_i$ has multiplicity 3 along five of the lines $l_1'\ldots l_6'$ and multiplicity 4 along the sixth.
 Thus, we obtain that for each $i=1,\ldots,6$ there is a unique surface $D_i\subset \P^3$
   of degree $12$ vanishing along $5$ general lines to order $3$ and along the sixth general line
   to order $4$.

   Similarly to \eqnref{eq:intersection product}, and using the same notation, the intersection product on $X$ satisfies
     \begin{equation}\label{eq:intersection product2}
   H^3 =1, H E_i^2=-1 \mbox{ and }  E_i^3 =-2
   \end{equation}
   with all other monomial triple intersections of $E_i$ and $H$ being 0.


Using an argument similar to the one used above, the Todd transformation corresponds to a linear transformation $\Gamma_{\calt}$ with:
   \begin{equation}\label{eq:trafo rules for Todd}
   \begin{array}{ccl}
      E_i' & \mapsto & 12H -3\E -E_i,\\
      H' & \mapsto & 19H -5\E,
   \end{array}
   \end{equation}
   which preserves triple products (\ref{eq:intersection product2}) and which is self inverse.

\subsection{Limits of fat disjoint lines}

The Cremona transformations above will be used to show nonemptiness of the systems in Theorem \ref{ex:Singular list}.
To show uniqueness of each surface, we shall rely on semicontinuity proving uniqueness in some particular position of the lines.
In some cases the chosen special position involves some of the lines becoming coplanar -- so they intersect; in that case the limit linear system acquires a higher multiplicity at the point of intersection of these lines:

 \begin{lemma}\label{lem:limit-fat-lines}
 	Let $l$ be a fixed line and let $r_t, t\in \Delta\subset \C$ be an analytic family of lines in $\P^3$ where $\Delta$ is a disk around 0, such that $l$ and $r_t$ are skew for $t\neq 0$ and the lines $l$ and  $r_0$ intersect at a point $p$. Let $m,n$ be positive integers.
 	If $F\in \C[[t]][X,Y,Z,W]$ is the equation of an analytic family of surfaces, such that for every $t\in \Delta\setminus 0$, $F_t$ has multiplicity at least $m$ (respectively $n$) along $l$ (respectively $r_t$), then $F_0$ has multiplicity at least $m$ (respectively $n$) along $l$ (respectively $r_0$), and multiplicity at least $m+n$ at $p$.
 \end{lemma}

 \begin{proof}
 	Without loss of generality we may assume that $l$ is the line $X=Y=0$ and $r_t$ is $Y-tW=Z=0$, so $p=[0:0:0:1]$.
 	The only claim that needs proof is the multiplicity at $p$, which is a local issue, so we will work in the regular local ring $R=\C[[t]][x,y,z]_p$, where $x=X/W, y=Y/W, z=Z/W$ and the lines are given by the ideals $I_l=(x,y)$ and $I_{r_t}=(y-t,z)$.
 	Then it is easy to see that $R/I_l^m$ and $R/I_{r_t}^n$ are Cohen-Macaulay modules, and $\dim R/I_l^m+\dim R/I_{r_t}^n=4=\dim R$.
 	Hence by \cite[V.B.6, Corollary after Theorem 4]{Ser}
 	one has $\operatorname{Tor}^R_1(R/I_l^m, R/I_{r_t}^n)=0$.
 	Therefore, tensoring the exact sequence $0\rightarrow I_l^m\rightarrow R\rightarrow R/I_l^m\rightarrow0$ with $R/I_{r_t}^n$ we obtain the exact sequence
 	$$0 \rightarrow I_l^m/(I_l^m\cdot I_{r_t}^n) \rightarrow R/I_{r_t}^n \rightarrow R/(I_l^m+I_{r_t}^n)\rightarrow0.$$
 	Then injectivity of the first map gives
 	$$I_l^m\cap I_{r_t}^n=I_l^m\cdot I_{r_t}^n$$
 	in $R$. Any $F$ as in the claim of course belongs to $I_l^m\cap I_{r_t}^n$, and it follows that it belongs to $I_l^m\cdot I_{r_t}^n$. In particular,
 	$F_0\in I_l^m\cdot I_{r_0}^n$ has multiplicity at least $m+n$ at $p$.
 	%
 \end{proof}

\subsection{First part of the proof of main theorem}\label{sec:first part}
   We are now in position to prove part of Theorem \ref{ex:Singular list}.
\proofof{Theorem \ref{ex:Singular list}, part I}
   We see immediately from \eqnref{eq:trafo rules for Todd}
   that the duodecic vanishing to order $3$ along $5$ general
   lines and to order $4$ along the sixth line
   is the image under $\calt_0$ of the exceptional divisor $E_i$
   of $\pi$. In particular it is a unique and irreducible member
   of the (projective) linear system
   in Theorem \ref{ex:Singular list} B). Moreover surfaces of this
   kind are rational.
   We will show in section \ref{sec:Waldschmidt} that these surfaces allow us
   to extend the list of known Waldschmidt constants of configurations
   of general lines in $\P^3$ obtained in \cite[Proposition B.2.1]{DHST14}.

   The Todd transformation explains also the linear system
   in Theorem \ref{ex:Singular list} D). Indeed, it is easy to check
   with the rules stated in \eqnref{eq:trafo rules for Todd} that
   the system $20H-6\E+5E_6$ is the image under $\Gamma_{\calt}$
   of the system $8H-\E-5E_6$, which is nonempty by \eqref{eq:virtdim}.
   To check that $8H-\E-5E_6$ has (projectively) a single element let $\Pi\subset X$ be the strict transform of a general plane through $\ell_6$, and consider the residual exact sequence
   \[0\rightarrow\mathcal{O}_X(7H-\E-4E_6)\rightarrow\mathcal{O}_X(8H-\E-5E_6)\rightarrow\mathcal{O}_\Pi((8H-\E-5E_6)|_\Pi)\rightarrow0.\]
   Since the system restricted to the plane $\Pi$ is the planar system of curves of degree two with five general base points, which consists of the unique conic through the five points, it will be enough to show that $7H-\E-4E_6$ is empty.
   Specialize the lines so that for $i=1,\dots,4$, $l_i$ and $l_{i+1}$ intersect at a point $P_i$.
   By Lemma \ref{lem:limit-fat-lines}, the limit of $7H-\E-4E_6$ consists of surfaces which have multiplicity 2 at the points $P_i$.
   Call $\Pi_i$ the strict transform of the plane $P_i\vee l_6$ and $\Pi_i'$ the strict transform of the plane $l_i\vee l_{i+1}$ (where for  projective subspaces  $A,B\subset \mathbb{P}^n$, $A\vee B$ denotes the span of $A$ and $B$).
   
   The restriction of the specialized system to $\Pi_i$ consists of septics containing $5l_6$, double at $P_i$ and with three further points, which are not aligned if the lines are general with the given restrictions; so this system is empty, which means that the four planes $\Pi_i$ are fixed components of the specialized system.
   The residual consists of cubic surfaces passing through the six (special) lines.
   It is not hard to see that the restriction of this residual to the planes $\Pi_1'$ and $\Pi_4'$ is noneffective, so these are also fixed components.
   What remains is the system $\call_1(1^{\times2})$, obviously empty.
   By semicontinuity, $7H-\E-4E_6$ is also empty for lines in general position.

   Finally, the linear system in Theorem \ref{ex:Singular list} A)
   is explained by the means of the cubo-cubic transformation $\tcalc$
   based at the first $4$ fat lines. To this
   end we show that already the linear system $10H-3\E$ is special.
   Indeed, we have $\expdim(10H-3\E)=54$ but the actual dimension
   is $56$ since this system is the image of the system
   $6H-\E$, which has the expected dimension $56$ and is non special
   by the aforementioned result of Hartshorne and Hirschowitz.
   Then, vanishing along an additional line imposes at most $11$ conditions,
   so that the system $10H-3\E-E_5-\dots-E_9$ has dimension at least  $1$.
   To show that it is exactly 1, we exhibit again a specialization for which the dimension is 1.
   Denote, $Q_i$ for each $i\in \{1,2,3,4\}$ the unique quadric containing all lines $l_j$ with $j\in \{1,2,3,4\}, j\ne i$.
   Let $\{P_1^1,P_2^1\}=Q_1\cap l_5$, $\{P_1^2,P_2^2\}=Q_2\cap l_5$, and $\{P_1^3,P_2^3\}$ two points in $Q_3$.
   Specialize the four last lines so that
   \begin{align*}
   l_6&=P_1^1 \vee P_1^3, & l_7&=P_2^1 \vee P_1^3,\\
   l_8&=P_1^2 \vee P_2^3, & l_9&=P_2^2 \vee P_2^3.
   \end{align*}
  
   By Lemma \ref{lem:limit-fat-lines}, the limit system has multiplicity 2 at each point $P_i^j$.
   If the points $P_i^3$ are general, the restriction of the limit system to $Q_i\cong \P^1\times\P^1$ consists of curves of bidegree $(10,10)$ containing three triple sections $l_i^3, i=2,3,4$ of type $(3,0)$ with two triple points at $l_1\cap Q_1$, which by B\'ezout forces the sections of type $(0,1)$ through these two points to split twice each; the residual are curves of bidegree $(1,6)$, with two double points $P_1^1$, $P_2^2$ and passing through 8 additional general simple points; this is known to be empty, see \cite{Lenarcik}. The same analysis applied to $Q_2, Q_3$ shows that the three quadrics $Q_1, Q_2, Q_3$ are fixed components of the system.
   The residual system consists of surfaces of degree 4 passing with multiplicity 1 through all lines except $l_4$; it is not hard to see that the only such surface splits as
   $Q_4 + (P_1^1\vee P_2^1 \vee P_1^3)+(P_1^2\vee P_2^2 \vee P_2^3)$.
   	
   The remaining case C) is the most interesting and it is dealt with in the next section.
\endproof

\begin{remark}\label{SpecialByCremona} Note that in the proof above we obtained special systems by  applying
cubo-cubic or Todd transformations to particular divisors.  Similarly, by applying them to
	$aH-(E_1+\dots+E_4)$ or to $aH-(E_1+\dots+E_6)$ respectively, we obtain  systems
	$(3a-8)H-(a-3)(E_1+\dots+E_4)$ and  $(19a-72)H-(5a-19)(E_1+\dots+E_6)$. For $a$ big enough
	the resulting systems have smaller expected dimensions than their actual  dimensions,
	thus producing many examples of unexpected surfaces.
\end{remark}

\section{The linear system $\call_{12}(3^{\times 6},2)$}\label{sec:7 lines}
   This section is devoted to the system C) in Theorem \ref{ex:Singular list}.

   For the system
   $\call=\call_{12}(3^{\times 6},2)$ we have
   $$\virdim(\call)=-2, \mbox{ hence } \expdim(\call)=0.$$

   We will now show that nevertheless the projectivisation of this system is non-empty and contains
   a single irreducible element (and no other elements). We do not see how to show the speciality of this system
   using birational transformations of the ambient space, so we take a different approach.

\proofof{Theorem \ref{ex:Singular list}, part II}
   Let $l_1,\ldots,l_7$ be general lines in $\P^3$ and
   let $f:X\to\P^3$ be the blowup of $\P^3$ along the first six lines with
   exceptional divisors $E_1,\dots,E_6$, respectively. As usual
   we write $\E$ for the union of the exceptional divisors $E_1,\dots,E_6$
   and denote by $H$
   the pullback of the hyperplane bundle to $X$. Then
   $K_X=-4H+\E$.
   We study the morphism defined by the anti-canonical system
   $$M=-K_X=4H-\E.$$			
   The divisors in this system correspond to quartics in $\P^3$
   vanishing along the first six lines. By the aforementioned result
   of Hartshorne and Hirschowitz we have
   $$h^0(X,M)=5,$$
   hence $M$ defines a rational map $\varphi_M: X\to \P^4$. The map is a morphism; one sees this by looking at
   reducible quartics containing the six lines (namely, products of two quadrics, each containing three of the six lines), and concluding that a base point on the blowup of one of the six lines
   implies that there would be a transversal to all six general lines.
   Note that $\varphi_M$ contracts  lines transversal to any $4$ of the $6$ given lines (there are $2\binom{6}{4}=30$ such contracted lines).
   By \cite{Vaz01} (or by computer calculations) the  image $Y$ of $\varphi_M$ is a quartic hypersurface, and $\varphi_M$ is  generically $1:1$.


\subsection{Existence}

   Let now $C\subset \P^4$ be the image of the seventh line under $\varphi_M$.
   Then $C$ is a rational normal curve of degree $4$. It is well known
   that its secant (chordal) variety is a determinantal threefold $T$ of degree $3$
   in $\P^4$, singular along $C$, see for example \cite[Proposition 9.7]{Har92}.

   Note that $T\cap Y$ is an element in $\calo_{\P^4}(3)\restr{Y}$, so  it pulls back  to a surface $D$
   of degree $12$ in $\P^3$, which vanishes to order $3$ along the first six
   lines.

   Moreover, since $T$ is singular along $C$ and $C\subset T\cap Y$ we conclude
   that $D$ is singular along $L_7$. Hence $D$ has vanishing orders along the lines $l_i$ with $i=1,\ldots,7$
   corresponding exactly to those required for part C) of Theorem \ref{ex:Singular list}.
   Thus the existence of duodecics $D$ vanishing to order $3$ along six general lines
   and singular along the seventh line is established.

   \subsection{Uniqueness}

   To show that the surface is unique we apply semicontinuouity, proving it in the particular case when the lines $l_1, \dots, l_5$ are chosen to be disjoint lines in a general cubic surface, and $l_6, l_7$ are general.

   A general cubic $\Sigma$ in \(\mathbb P^3\) can be understood as the blowup $\sigma:\Sigma\rightarrow\P^2$ of \(\mathbb P^2\) at 6 general points, $A_1,\ldots,A_6$, embedded by the system $\mathcal{O}_{\mathbb{P}^2}(3)\otimes I_{A_1}\dots \otimes I_{A_6}$. The pullback of
   $\mathcal{O}_{\mathbb{P}^2}(1)$ on the cubic surface is denoted by $h$.
   The $27$ lines on $\Sigma$  are identified as the $6$ exceptional divisors $a_1, \dots, a_6$, the $\binom62=15$ lines through pairs of points, and the $6$ conics through $5$ out of $6$ points.
   We choose $l_1, \dots, l_5$ to be five of these latter lines.

   Denote by $\pi: W\to \mathbb P^3$ the blowup of the 7 lines, $l_1,\dots, l_7$,  (in this particular position) with exceptional divisors $E_1,\dots, E_7$. Let $S\subset W$ (resp. $N\subset W$) be
   the proper transform of the cubic (resp. of the duodecic $D$).  In $S$ the divisor $E_i|_S$ has class $2h-\mathbbm{a}+a_i$ for $i=1,\dots,5$ (where as usual $\mathbbm{a}=a_1+\dots+a_6$). $S$ is the blowup $\pi: S\to \Sigma$  of the cubic at 6 additional points, $F_1, F_2, F_3, G_1, G_2, G_3$, namely its intersection points with $l_6$ and $l_7$. Denote the exceptional curves of this blowup by $f_1,f_2,f_3,g_1,g_2,g_3$ respectively, and set $\mathbbm{f}=f_1+f_2+f_3$, $\mathbbm{g}=g_1+g_2+g_3$. The fact that these points $S\cap(l_6\cup l_7)$ lie on two lines translates to
   \begin{equation}\label{eq:pencils}
   h^0(S,\mathcal O_{S} (3h-\mathbbm{a}-\mathbbm{f}))=
   h^0(S,\mathcal O_{S} (3h-\mathbbm{a}-\mathbbm{g}))=2.
   \end{equation}
   \begin{remark}\label{rm:no-three-collinear}
   	We claim that for a general choice of the lines $l_6, l_7$, no three of the image points of $F_1, F_2, F_3, G_1, G_2, G_3$ in $\P^2$ are collinear.
   	To prove this, first observe that $F_1, F_2$ (resp. $G_1, G_2$) can be chosen arbitrarily, and then $F_3$ (resp. $G_3$) is determined as the third intersection of $l_6=F_1\vee F_2$ (resp. $l_7=G_1\vee G_2$) with the cubic surface $\Sigma$.
   	So we can assume that the image points of $F_1, F_2$ (resp. $G_1, G_2$) in $\P^2$ are not aligned with any $A_i$.
   	This already implies that the images of $F_1, F_2, F_3$ (or equivalently $G_1, G_2, G_3$) are not collinear, because if they were, by \eqref{eq:pencils} there would be at least one section of $\mathcal{O}_S(3h-\mathbbm{a}-\mathbbm{f})$ vanishing on the line containing them, and then the six points $A_1, \dots, A_6$ would belong to a conic, which contradicts their being general points.
   	Now fix a choice of $F_1, F_2, F_3$ and let $t\subset \Sigma$ be the pullback of the triangle through their images in $\P^2$.
   	Consider the rational map
   	\begin{align*}
   	\tau:\Sigma \times \Sigma & \dashrightarrow \Sigma \\
	(G_1,G_2) & \mapsto G_3   	
   	\end{align*}
   	which is clearly dominant.
   	If $U$ is the open subset of $\Sigma\times\Sigma$ where $\tau$ is defined, then choosing $(G_1,G_2)\in U\setminus ((t\times \Sigma) \cup (\Sigma\times t) \cup \tau^{-1}(t))$ guarantees that the images of $F_i, F_j, G_k$ in $\P^2$ are not aligned.
   	By symmetry, a general choice of $l_6$ and $l_7$ will give that the images of $G_i, G_j, F_k$ are not aligned either.
   \end{remark}

    Now consider the residual exact sequence:
      \[0 \to \mathcal{O}_{W}(N-S)\to \mathcal{O}_{W}(N) \to \mathcal{O}_S(N|_S) \to 0\]
   We need to see that the global sections of the sheaf in the middle have dimension (at most) 1; we do so by proving that the global sections of the two other sheaves have dimensions 0 and 1 respectively.

   The restriction of the class of the duodecic to S is:
\begin{multline}
   N|_S=12H|_S-3\E|_S+E_7|_S=\\
   12(3h-\mathbbm{a})-3\sum_{i=1}^5(2h-\mathbbm{a}+a_i)-3\mathbbm{f}-2\mathbbm{g}=\\
   6h+3a_6-3\mathbbm{f}-2\mathbbm{g}.
\end{multline}


   Thus $a_6$ (which as a line in $\P^3$ is the unique common transversal to $l_1,\dots,l_5$) is a fixed component of the restricted system $N|_S$.
   After subtraction of this fixed part the residual corresponds to the planar system $\mathcal{O}_{\mathbb{P}^2}(6)\otimes I_{F_1}^3\otimes I_{F_2}^3\otimes I_{F_3}^3\otimes I_{G_1}^2\otimes I_{G_2}^2\otimes I_{G_3}^2$
   (by a slight abuse of notation we identify the points $F_1,F_2, F_3, G_1, G_2, G_3$ with their images under $\sigma$ in $\mathbb{P}^2$). Any sextic in this system splits by the B\' ezout theorem into the sum of three conics: $C_i$ through $F_1, F_2, F_3, G_j, G_k$ for $i=1,2,3$ and $j,k$ such that $\{i,j,k\}=\{1,2,3\}.$
   Note that these conics are irreducible by Remark \ref{rm:no-three-collinear}.
   Thus, the dimension of this system is $1$ and consequently $h^0(S,\calo_S(N|_S))=1$.
 It remains to see that $h^0(W,\mathcal{O}_W(N-S))=0$, which we prove in Proposition \ref{pro:nonic-empty} below.
\endproof

    \begin{proposition}\label{pro:nonic-empty}
	Let $l_1, \dots, l_5\subset\P^3$ be disjoint lines in a general cubic surface, and $l_6, l_7$ general lines in $\P^3$.
	Let $W\rightarrow\P^3$ be the blowup of the seven lines, and denote $E_i$ the exceptional divisors, $\E=E_1+\dots+E_7$ as above.
	Then $H^0(\mathcal{O}_W(9H-2\E-E_6))=0$.    	
    \end{proposition}
    \begin{proof}
    	
    	By projection formula we have $$H^0(W, \calo_W(9H-2\E-E_6))=H^0( \P^3, \calo_{\P^3}(9) \otimes I_{l_1}^2 \dots  \otimes  I_{l_5}^2\otimes  I_{l_6}^3 \otimes I_{l_7}^2).$$
    	
    	We specialize further $l_7$ to $l_7'$, which is the line on $\Sigma$, the proper transform of the line $A_1\vee A_2$ in $\P^2$. This line intersects $l_1$ and $l_2$ and no other line $l_i$. Let $P_1=l_1\cap l_7'$ and $P_2=l_2\cap l_7'$.
    	
		By semicontinuity and Lemma \ref{lem:limit-fat-lines}, it will be enough to show that there is no surface $T$ of degree 9 in $\mathbb{P}^3$ which is singular along all 7 lines, has multiplicity 3 along $l_6$, and multiplicity 4 at $P_1$ and $P_2$.
		The restriction of such a surface to $\Sigma$ would be, in the notations as above,
		\begin{multline*}
		    T|_{\Sigma}=9(3h-\mathbbm{a})-2\sum_{i=1}^5(2h-\mathbbm{a}+a_i)-2(h-a_1-a_2)-3\mathbbm{f}=\\
			5h+a_1+a_2-(a_3+a_4+a_5)+a_6-3\mathbbm{f}.
		\end{multline*}		
After taking out the fixed components \(a_1+a_2+a_6\) we are left with the planar system $\calo_{\P^2}(5)\otimes I_{F_1}^3\otimes I_{F_2}^3\otimes I_{F_3}^3\otimes I_{A_1}\otimes I_{A_2}\otimes I_{A_3} $, which by Remark \ref{rm:no-three-collinear} is non-effective. Therefore $\Sigma$ must be a component of $T$.
		
		Let $U=T-\Sigma $, which would be a surface of degree 6 passing through all 7 lines, with multiplicity 3 along $l_6$, and multiplicity 3 at $P_1$ and $P_2$.
		Let
		\(\Pi_1, \Pi_2\) be the two planes $\Pi_i=l_6\vee P_i, \ i=1,2$.
		The restriction $U|_{\Pi_i}$ is a plane curve of degree $6$ containing $l_6$ as a triple component, and vanishing at $P_i$ to order $3$ and vanishing in $4$ additional points, which are intersection points $\Pi_i\cap l_j$ for $j\neq 6$ and such that $P_i\notin l_j$. Since such a curve does not exist, $\Pi_1$ and $\Pi_2$ must be components of $U$.
		
		Let $V=U-\Pi_1-\Pi_2$ be a surface of degree 4 passing through all 7 lines, and singular at $P_1$ and $P_2$.
		A similar computation as above shows that $V|_{\Sigma}$ is a system of divisors $h+a_1+a_2+a_6$ passing through $F_1,F_2,F_3$, which is non-effective, so $\Sigma$ is again a component of $V$. But $V-\Sigma$ would then be a plane containing $l_6$, $P_1$, and $P_2$, which is clearly impossible, and we are done.
    \end{proof}

    \begin{lemma}\label{lem:resolution}
    	Let $l_1, \dots, l_7\subset\P^3$ be general lines in $\P^3$, let $W\rightarrow\P^3$ be the blowup of the seven lines, and denote $N \subset W$ the pullback of the unique surface in $\mathcal{L}_{12}(3^{\times 6},2)$.
    	Then $N$ is smooth.    	
    \end{lemma}
    \begin{proof}
		By semicontinuity of multiplicities 
		it is enough to check smoothness for a particular choice of lines.
		We do this computationally, see the Appendix.
    \end{proof}

    \begin{corollary}\label{cor:general-type}
    	Let $D$ be the unique surface in $\mathcal{L}_{12}(3^{\times 6},2)$, and $N\subset W$ its smooth model. Then $D$ is a surface of general type, with $p_g=6$, $q=0$, and $K_N^2=8$.
    \end{corollary}
    \begin{proof}
    	Denote $\pi_6:X\rightarrow\P^3$ the blowup of the first 6 lines, and $\tau :W\rightarrow X$ the blowup of $l_7$.
    	By Lemma \ref{lem:resolution} $N$ is the strict transform of $D$ in $W,$ and by adjunction the canonical class of $N$ is given as $K_N=(K_W+N)|_N$.
    	Since $h^0(\mathcal{O}_W(K_W))=h^1(\mathcal{O}_W(K_W))=0$, $p_g(N)=h^0(\mathcal{O}_W(K_W+N))$.
    	This can be computed, with the help of the map $\phi_M$ introduced in the proof of the exitence of $D$.
    	Indeed, $K_W+N=-4H+\E+12H-3\E+E_7=-2K_W+E_7=\pi^*(2M)-E_7,$ so
    	\begin{equation*}
    	p_g(N)=h^0(W,\mathcal{O}_W(\pi^*(2M)-E_7)=
    	h^0(X,\mathcal{O}_X(2M)\otimes \mathcal{I}_{L_7})=
    	h^0(Y,\mathcal{O}_{Y}(2)\otimes \mathcal{I}_{C}),  	
    	\end{equation*}
    	where as before $Y=\phi_M(X)$ is a quartic threefold in $\P^4$, and $\mathcal{I}_C$ is the ideal sheaf of the quartic curve $C=\phi_M(L_7)$.
    	Now the residual exact sequence
    	\[0\rightarrow\mathcal{O}_{\P^4}(-2)\rightarrow\mathcal{O}_{\P^4}(2)
    	\rightarrow\mathcal{O}_{Y}(2)\rightarrow 0\]
    	gives
    	\[p_g(N)=h^0(\P^4,\mathcal{O}_{\P^4}(2)\otimes\mathcal{I}_C)=\dim_\C(I_C)_2,\]
    	where $I_C$ is the homogeneous ideal of $C$ in the coordinate ring of $\P^4$, and $(I_C)_2$ is the degree 2 piece.
    	The ideal of a rational normal quartic curve is well known: it is generated by 6 independent quadrics \cite[Examples 1.14 and 9.3]{Har92}. So $p_g(N)=6$.    	 
    	The selfintersection of the canonical divisor is computed as the intersection number of three divisor classes on $W$, which by \eqref{eq:intersection product2} is
    	\[
    	K_N^2=(K_W+N)^2\cdot N =(8H-2\E+E_7)^2\cdot(12H-3\E+E_7)=8.
    	\]
    	So we have $p_g(N)>0, K_N^2>0$, and hence $N$ is a surface of general type.
    	On the other hand, the exact sequence
    	\[0 \rightarrow \mathcal{O}_W(2K_W+N) \rightarrow \mathcal{O}_W(2K_W+2N) \rightarrow \mathcal{O}_N(2K_N)\rightarrow 0\]
    	tells us that $\chi(N,\mathcal{O}_N(2K_N))=\chi(W,\mathcal{O}_W(2K_W+2N))-\chi(W,\mathcal{O}_W(2K_W+N))$, and these two Euler characteristics can be computed as virtual dimensions with \eqref{eq:virtdim}:
    	\begin{align*}
    	\chi(W,\mathcal{O}_W(2K_W+2N))&=\chi(W,\mathcal{O}_W(16H-4\E+2E_7))=
    	\dim_{\mathrm{vir}}(\mathcal{L}_{16}(4^{\times6},2))=20\\
    	\chi(W,\mathcal{O}_W(2K_W+N))&=\chi(W,\mathcal{O}_W(4H-\E+E_7))=
\dim_{\mathrm{vir}}(\mathcal{L}_{4}(1^{\times6}))=5\\
    	\end{align*}
    	Therefore, by Riemann-Roch we have
    	$$15=\chi(N,\mathcal{O}_N(2K_N))=\frac12((2K_N)^2-2K_N\cdot K_N)+1-q+p_g=8+1-q+6,$$
    	so the irregularity vanishes, and $\chi(\mathcal{O}_N)=7$.
    \end{proof}

\section{Waldschmidt constants}\label{sec:Waldschmidt}
   The rest of this note is devoted to asymptotic invariants
   of the homogeneous ideal $I=I(Z_s)$ of $s$ reduced  general lines in $\P^3$.
Recall that $I=\oplus (I)_d$ and that the \emph{initial degree}
   of $I$ is defined as the number
   $$\alpha(I)=\min\left\{d:\; (I)_d\neq 0\right\}.$$

   The $m$-th symbolic power  in this situation is  $$I^{(m)}=\bigcap I(L_i)^m.$$

  The  asymptotic counterpart of $\alpha(I)$ is the \emph{Waldschmidt constant} of $I$,
   defined as
   $$\alphahat(I)=\inf\frac{\alpha(I^{(m)})}{m}=\lim\limits_{m\to\infty}\frac{\alpha(I^{(m)})}{m}.$$
   These constants were first studied by Waldschmidt \cite{Wal77} for ideals
   defining finite sets of points in  $\P^n$.
   They are very hard to compute in general.
   For ideals of general lines in $\P^3$, Waldschmidt constants were studied
   in \cite{GHVTb13, DHST14, DFST} and computed for up to $5$ lines; see
   \cite[Corollary 1.1]{GHVTb13} for up to 4 lines or \cite[Proposition B.2.1]{DHST14}.
   In this section we extend these results to $6$ lines. We expect that the value
   of the Waldschmidt constant for $7$ lines is determined by the duodecics
   of type C in Theorem \ref{ex:Singular list} and that
   Waldschmidt constants of a greater number of lines in $\P^3$ are governed by
   \cite[Conjecture 5.5]{GHVT13} (see \cite[Conjecture A]{DHST14} for a generalization).

\begin{proposition}
   Let $Z_6$ be the union of $6$ general lines in $\P^3$. Then
   $$\alphahat(I)=\frac{72}{19}$$
   for $I=I(Z_6)$.
\end{proposition}
\proof
   Let $I$ be the
   homogeneous ideal of the union of six general lines $L_1,\ldots,L_6$.
   By Theorem \ref{ex:Singular list} for each $i=1,\ldots,6$, there exists
   a duodecic $D_i$ vanishing along $L_j$, $j\neq i$, to order $3$
   and along $L_i$ to order $4$. Then the symmetrized divisor
   $D=\sum_{i=1}^6 D_i$ has degree $72$ and vanishing order $19$
   along all lines. Hence
   $$\alphahat(I)\leq  \frac{72}{19}.$$  
   In order to see the reverse inequality, it suffices to show
   that $h^0(\calo_{\P^3}(d)\otimes I^{(m)})=0$,
   whenever $\frac{d}{m}<\frac{72}{19}$. In fact, it suffices
   to show that
   $$h^0(\calo_{\P^3}(72m-1)\otimes I^{(19m)})=0$$
   for all $m\geq 1$.
   Here the Todd transformation again comes into the picture.
   We check easily, that the Todd transformation
   of the system $(72m-1)H-19m\E$ (in the notation from section \ref{sec:Cremona})      
   is $-19H+m\E$ which is obviously non-effective.
\endproof
   The reader might find it convenient to have an overview of known
   Waldschmidt constants for $s\leq 6$ general lines in $\P^3$ as well
   as the expected value for $s\geq 7$.
   \renewcommand{\arraystretch}{1.5}

   \begin{center}
\begin{tabular}{|c||c|c|c|c|c|c|}
   \hline
   s  & 1 & 2 & 3 & 4 & 5 & 6\\
   \hline
  $\alphahat(I_s)$ & 1 & 2 & 2 & $\frac{8}{3}$ & $\frac{10}{3}$ & $\frac{72}{19}$\\
  \hline
\end{tabular}
   \end{center}
   For $s=7$ we expect that $\alphahat(I_7)=21/5$ and
   for $s\geq 8$ we conjecture that $\alphahat(I_s)$ equals the largest real root
   of the polynomial
   $$\tau^3-3s\tau+2s.$$
   This is inspired by the conjectures given in \cite{GHVT13, DHST14}, but
   it is stronger, because it is specific about the values of $s$ for which it is conjectured to hold.


\bigskip
\paragraph*{\emph{Acknowledgement.}}
Harbourne was partially supported by Simons Foundation grant \#524858.
Ro\'e was partially supported by Spanish Mineco grant  MTM2016-75980-P and by Catalan AGAUR grant 2017SGR585.
Tutaj-Gasi\'nska was partially supported by
   National Science Centre grant 2017 /26/M/ST1/00707.
   This paper has also benefitted from discussions with various mathematicians:
   Ciro Ciliberto, John Lesieutre, Juan Migliore, Rick Miranda, Uwe Nagel, John Ottem, David Schmitz, Damiano Testa, Duco van Straten.
   This list is by no means complete and we apologize in advance to those who are not mentioned.
   %


\section*{Appendix}
\begin{verbatim}
//This is a text file, to be made available in the authors' 
//pages, containing scripts to 
//check some of the claims made in the paper, 
//in particular that the surface N in section 5
//is smooth.
//These scripts can be run in 
//Singular, https://www.singular.uni-kl.de/

//We work over a field of large characteristic, 
//althoug mainly interested in CC
//Since our purpose is to argue by semicontinuity, this is ok.

ring R=32003,(x,y,z,w),dp;        
LIB "primdec.lib"; option(redSB); 

//This procedure shows the degrees of the 
//generators of an ideal
//We use it to quickly check the dimension of the homogeneous 
//component in its initial  degree, 
//in particular for uniqueness of the 12-ic
proc writedegrees(ideal CC) {
	"generators degrees:";
	string ss="";
	int i;
	for (i=1;(i<=size(CC))&&(i<=500);i++) {
		ss=ss+string(deg(CC[i]))+" ";
	}
	ss;
}

//We are interested in sets of 7 general lines, 
//and we choose 7 "nice" lines to argue 
//by semicontinuity. 
//These turn out to be general enough for our purposes.
ideal L1=x,y;
ideal L2=z,w;
ideal L3=x+w,y+z;
ideal L4=x-y+w,2x+z;
ideal L5=2y+z+2w,x-z+w;
ideal L6=2y+z,2x-w;
ideal L7=x+z+2w,y+w;

//The rational map phi_M of section 5 
//is given by quartics through L1-L6. 
//We first check that there are 5 independent such quartics
ideal M1=intersect(L1,L2,L3,L4,L5,L6);
M1=std(M1);         
writedegrees(M1);
ideal M=M1[1],M1[2],M1[3],M1[4],M1[5]; 

//Next we check that the image of the map 
//is a quartic hypersurface
ring S=32003,(x,y,z,w,a,b,c,d,e),dp; 
ideal M=fetch(R,M);
ideal image=a-M[1],b-M[2],c-M[3],d-M[4],e-M[5];
image=std(image);
image=eliminate(image,xyzw);
image; //Output: a single polynomial of degree 4

//A line meeting this quartic properly will do so in 4 points.
//If its preimage consists of 4 points, 
//this will prove that the map is generically 1:1
setring R;
ideal philin=M[2],M[3],M[5]; //The preimage of line [a:0:0:b:0]
philin=std(philin);
philin=sat(philin,M)[1];
hilb(philin);     //consists of four pts 

//Next we check uniqueness of the 12-ic for these lines
ideal I=intersect(L1^2,L2^3,L3^3,L4^3,L5^3,L6^3,L7^3);
I=std(I);
writedegrees (I); 
poly D=I[1];      //Equation of the 12-ic

//And now we will check the singularities of N (lemma 5.3)
//First we want to see that all singularities lie on N ∩ E
//We do it by computing the jacobian ideal of D 
//and quotienting by the ideals of the lines
//with the corresponding multiplicity
ideal m=x,y,z,w;
ideal singD=jacob(D);
singD=std(singD);
singD=quotient(singD,L1); //These computations take some time
singD=quotient(singD,L1);  
singD=quotient(singD,L2); 
singD=quotient(singD,L2); 
singD=quotient(singD,L2); 
singD=sat(singD,m)[1];
singD=quotient(singD,L3^3);
singD=quotient(singD,L4^3);
singD=quotient(singD,L5^3);
singD=quotient(singD,L6^3);
singD=quotient(singD,L7^3);
singD; //Output: 1, so OK

//Now let us analyze singularities on the exceptional divisors
//Each E_i is isomorphic to P^1 x P^1, and E_i \cap N, 
//the tangent cone of N along l_i, 
//is given by a bihomogeneous polynomial of bidegree (12-m,m) 
//where m is the multiplicity  of D along l_i.
//We shall show that each of these tangent cones 
//is in fact smooth,
//so not only N->D is a  resolution of D, 
//it is an embedded resolution.
//By genericity of the lines and symmetry, 
//it is enough to do it for the double line 
//and one of the triple ones.
//By semicontinuity  of multiplicity, 
//it is enough to do it for these particular lines.
//The computation is especially simple 
//for the "coordinate" lines L1 and L2
poly D1=reduce(D,std(L1^3)); //The tangent cone to a double line
poly D2=reduce(D,std(L2^4)); //The tangent cone to a triple line

//We change the ordering to account 
//for bihomogeneous coordinates 
ring T=32000,(x,y,z,w),(dp(2),dp(2)); 
poly D1=fetch(R,D1); 
poly D2=fetch(R,D2);

ideal singD1=jacob(D1);
ideal m1=x,y;
ideal m2=z,w;
singD1=sat(singD1,m1)[1];
singD1=sat(singD1,m2)[1];
singD1; //Output: 1 so the surface is transverse to E1

ideal singD2=jacob(D2);
singD2=sat(singD2,m1)[1]; //Takes some time
singD2=sat(singD2,m2)[1];
singD2; //Output: 1 so the surface is transverse to E2

\end{verbatim}


\bigskip
   Marcin Dumnicki, Halszka Tutaj-Gasi\'nska,
   Jagiellonian University, Institute of Mathematics, {\L}ojasiewicza 6, PL-30-348 Krak\'ow, Poland

\nopagebreak
   \textit{E-mail address:} \texttt{Marcin.Dumnicki@im.uj.edu.pl}

   \textit{E-mail address:} \texttt{Halszka.Tutaj@im.uj.edu.pl}

\bigskip
   Brian Harbourne,
   Department of Mathematics, University of Nebraska-Lincoln, Lincoln, NE, 68588, USA

\nopagebreak
   \textit{E-mail address:} \texttt{bharbourne1@unl.edu}

\bigskip
   Joaquim Ro\'e,
   Barcelona Graduate School of Mathematics and
   Departament de Matem\`atiques, Facultat de Ci\`encies,  Universitat Aut\`onoma de Barcelona, 08193 Bellaterra (Barcelona), Spain.

\nopagebreak
   \textit{E-mail address:} \texttt{jroe@mat.uab.cat}

\bigskip
   Tomasz Szemberg,
   Department of Mathematics,
   Pedagogical University of Cracow,
   Podchor\c a\.zych 2,
   PL-30-084 Krak\'ow, Poland.

\nopagebreak
   \textit{E-mail address:} \texttt{tomasz.szemberg@gmail.com}


\end{document}